\documentclass [11pt, a4paper] {article}
\usepackage{bbm}
\usepackage{amssymb,amsmath,amsthm}
\DeclareMathOperator{\supp}{supp}
\DeclareMathOperator{\diam}{diam}
\DeclareMathOperator{\dive}{div}

\newcommand\blfootnote[1]{%
  \begingroup
  \renewcommand\thefootnote{}\footnote{#1}%
  \addtocounter{footnote}{-1}%
  \endgroup
}
\newtheorem{theorem}{Theorem}
\newtheorem{lemma}[theorem]{Lemma}

\newtheorem{Lemma}[theorem]{Lemma}

\newtheorem{definition}[theorem]{Definition}

\newtheorem{remark}[theorem]{Remark}

\newcommand{\R}{{\mathbbm{R}}}
\usepackage[%
   bookmarks=true,
   unicode,
   plainpages=false,
   ]{hyperref}
\title{\bf Classical solutions of the divergence equation with Dini-continuous datum}
\author{Luigi C. Berselli and Placido Longo
  \\
  Dipartimento di Matematica
  \\
  Universit\`a di Pisa
  \\
  Via F. Buonarroti 1/c, 
  \\
  I-56127, Pisa, Italia.
  \\
 email: luigi.carlo.berselli@unipi.it, placido.longo@unipi.it}
\date{}
\begin{document}
\maketitle
\begin{abstract}
  We consider the boundary value problem associated to the divergence operator with
  vanishing Dirichlet boundary conditions and we prove the existence of classical
  solutions under slight assumptions on the regularity of the datum.
\end{abstract}
\section{Introduction}
In this paper we deal with the existence of classical solutions for the boundary value problem 
\begin{equation}
  \label{eq:div}
  \left\{
    \begin{aligned}
      \dive u&=F\qquad\text{in }\Omega,
      \\
      u&=0\qquad\text{on }\partial\Omega,
    \end{aligned}
  \right.
  \end{equation}
namely we look for solutions $u:\Omega\to\R^{n}$, such that $u\in C^{1}(\Omega)\cap
C^{0}(\overline{\Omega})$. Here $\Omega\subset\R^{n}$ is a smooth and bounded domain,
while $F$ is a given continuous function satisfying the compatibility condition
$\int_{\Omega}F(x)\,dx=0$. This is a classical problem in mathematical fluid mechanics,
strictly connected with the Helmholtz decomposition and the div-curl lemma (see Kozono and
Yanagisawa~\cite{KY2009b}). We recall that by dropping the boundary condition a solution
of the first order system~\eqref{eq:div} can be readily obtained by taking the gradient of
the Newtonian potential of $F$.  These aspects are extensively covered in
Galdi~\cite[Ch. 3]{Gal2011} with special attention to the work of
Bogovski\u\i~\cite{Bog1980}, where the problem is solved in the Sobolev spaces
$H_{0}^{1,p}(\Omega)$.  Further developments may also be found in Borchers and
Sohr~\cite{BS1990}. For different approaches and results the reader could consider the
books by Lady\v{z}henskaya~\cite{Lad1969} and Tartar~\cite{Tar1978}, especially regarding
the solution in the Hilbert case, while Amrouche and Girault~\cite{AG1994} used the
negative norm theory developed by Ne\v{c}as~\cite{Nec1967}.

Our approach follows closely the Bogovski\u\i's one, where the representation
formula~\eqref{eq:bogovskij} below, in analogy with the ``cubature'' formulae of Sobolev,
gives explicitly a special solution of the problem~\eqref{eq:div}, which \textit{per se}
has infinitely many solutions.  The formula~\eqref{eq:bogovskij} turns out to be extremely
flexible in the applications to many different settings as, for instance, the recent
results for weighted $L^{p(x)}$-spaces (see Huber~\cite{Hub2011}). Classical results in
H\"older spaces have been shown in Kapitansk\u\i\ and Piletskas~\cite{KP1984}, as a
corollary of a more general result, which seems to be obtained in a way different from
ours. For the H\"older case see also the recent review in Csat\'o, Dacorogna, and
Kneuss~\cite{CDK2012}.  In addition, we also note that the non-uniqueness feature of the
first order system~\eqref{eq:div} allows some existence results with more regularity than
expected from the usual theorems, as the striking results of Bourgain and
Brezis~\cite{BB2003}, coming from a non-linear selection principle of the solutions to a
linear problem (see the extensions to the Dirichlet problem in Bousquet, Mironescu, and
Russ~\cite{BMR2013}).

Our interest for the problem is twofold: on one side, we want to investigate the results
close to the limiting case $F\in L^\infty(\Omega)\cap C(\Omega)$, for which
counterexamples to the existence of a solution are known (see Bourgain and
Brezis~\cite{BB2003} and Dacorogna, Fusco, and Tartar~\cite{DFT2003}); on the other side,
we are interested in relaxing as much as possible the assumptions needed to get classical
solutions. This is motivated by the aim of finding weaker sufficient assumptions which
allow to construct classical solutions to fluid mechanics problems.  Since the continuity
of $F$ is not enough to that purpose, we went back to the very old results by
Dini~\cite{Dini1902} and Petrini~\cite{Pet1908} about the Poisson equation, and we
consider the problem with the additional hypothesis that $F$ is Dini-continuous (see below
for a definition). Our proof follows closely an argument used by Korn to obtain a similar
regularity result for the second derivative of the Newtonian potential (see Gilbarg and
Trudinger~\cite[Ch.~4]{GT1998}) and exploits the property of the Dini continuity to
``regularize'' the singularity in the second derivative of the potential, outside the
theory of Calder\`on-Zygmund of singular integral operators.  About the boundary
condition, our proof is based on some ``new'' insight of the formula (in the sense that we
made some simple observations on the Bogovski\u\i{} formula which we cannot find stated
explicitly elsewhere in literature) still valid when the datum $F$ cannot be approximated
with compactly supported smooth data. Such an approximation seems to play a fundamental
role in Sobolev or Orlicz spaces.

We wish to mention that the link between Dini continuity and existence of classical
solutions in fluid mechanics started with Shapiro~\cite{Sha1976} in the steady case and
found a very interesting application with the paper of Beir\~ao da Veiga~\cite{Bei1984},
where the 2D Euler equations for incompressible fluids are solved in a critical spaces for
the vorticity. More recently, the same results have been also employed by
Koch~\cite{Koch2002} and in~\cite{BB2015} to study fine properties of the long-time
behavior of the Euler equations. Finally, the interest for classical solutions of the
Stokes system has been revived in the recent papers of Beir\~ao da
Veiga~\cite{Bei2014,Bei2016}, that provided a further motivation to our analysis of the
divergence and curl operator, since they are among the building blocks of the theory. We
also point out that the system~\eqref{eq:div} is not elliptic, hence the well-known
results for elliptic equations and systems do not apply directly.

To conclude we also mention that the problem of classical solutions for the curl equation
(again with Dirichlet condition) will be treated elsewhere~\cite{BL2017b} following the
same method, by using similar (but more complicated) representation formula, known in
that case.
\section{Notation and preliminary results}
In this section we recall the main definitions we will use, as well as some basic facts
about the representation formula developed by Bogovski\u\i.  The results of this section are
well-known, but some of them on the role of the boundary condition are not explicitly
available in the literature.

In the following we denote by $B(x,R)=\{y\in\R^{n}:|y-x|<R\}$, the ball of radius $R$
centered at $x$, by
$S^{n-1}=\{y\in\R^{n}:|y|=1\}$ the unit sphere of $\R^{n}$, and by $|S^{n-1}|$ its
$(n-1)$-dimensional measure. 

We also denote by $C_{D}(\Omega)$ the space of the (uniformly) Dini-continuous functions
F, \textit{i.e.}, such that if one introduces the modulus of (uniform) continuity
\begin{equation*}
\omega(F,\rho)=\sup_{\stackrel{x,y\in\Omega} {|x-y|<\rho}}|F(x)-F(y)|,
\end{equation*}
the function $\omega(F,\rho)/\rho$ is integrable around $0^{+}$. We equip the space of
Dini continuous functions with the following norm
\begin{equation*}
  \|F\|_{C_D}=\max_{x\in
    \overline{\Omega}}|F(x)|+\int_0^{\diam(\Omega)}\frac{\omega(F,\rho)}{\rho}\,d\rho, 
\end{equation*}
and it turns out to be a Banach space. We remark that, by the uniform continuity, any
function in $C_{D}(\Omega)$ may be extended up to the boundary of $\Omega$ with the same
modulus of continuity.  We observe also that $C^{0,\alpha}(\Omega)\subset C_D(\Omega)$ for
all $0<\alpha\leq1$ and recall that its relevance in partial differential equations comes
from the result that, if $f\in C_D(\Omega)$, then the solution of the Poisson equation
\begin{equation*}
  \Delta u=f
\end{equation*}
with zero Dirichlet conditions satisfies $D^2u\in C(\Omega)$ (see for instance Gilbarg and
Trudinger~\cite[Pb.~4.2]{GT1998}.
\subsection{Bogovski\u\i's formula and its variants}
The aim of this section is to provide a representation formula for a solution of the
divergence problem, due to Bogovski\u\i~\cite{Bog1980}, as well as several useful variants
and consequences. 

Unless differently specified (namely, in the last section), the following hypotheses will
be tacitly assumed throughout all the paper. Let $B$ denote the open unit ball of
$\R^{n},\,n\geq 2$, centered at the origin.  The scalar function $\psi\in
C^{\infty}_{0}(\R^{n})$ with $\supp\psi\subseteq B$, and it is not vanishing identically.
We will denote by $\partial_{j}\psi$ the partial derivative of $\psi$ with respect to its
$j^{th}$ argument. The domain $\Omega$ will be a bounded open subset of $\R^{n}$,
star-shaped with respect to any point of $\overline{B}$.

The main results to be proved are the following Theorems.
\begin{theorem}
  \label{thm:theorem1}
  Assume the previous hypotheses. Let $q>n$ and let $F\in L^{q}(\Omega)$. Then:
  \begin{enumerate} 
\item The Bogovski\u\i's formula
    \begin{equation}
      \label{eq:bogovskij}
      v(x)=\int_{\Omega}F(y)\left[\frac{x-y}{|x-y|^{n}}\int_{|x-y|}^{+\infty}\psi\left(y+
          \xi\frac{x-y}{|x-y|}\right)\xi^{n-1}d\xi\right]\,dy, 
    \end{equation}
    defines for any $x\in\R^{n}$ (and not only almost everywhere) a function
    $v:\,\R^{n}\to\R^n$;

  \item $\quad v(x)=0\quad \forall\, x\in\R^n\backslash\Omega$;
  \item For any $q>n$
    \begin{equation*}
      | \,v(x)\,|\leq  c\;\| F\|_{L^{q}(\Omega)}\quad\forall x\in\R^{n},
    \end{equation*}
    where $ c$ depends only on $n$, $\psi$, $\diam\,\Omega$, and $q$;
  \item $$
    v(x)=\int_{\Omega}F(y)\left[(x-y)\int_{1}^{\infty}\psi\left(y+\alpha(x-y)\right)\alpha^{n-1}d\alpha\right]
    \,dy;$$
  \item $$ v(x)=\int_{\Omega}F(y)\left[\frac{x-y}{|x-y|^{n}}\int_{0}^{\infty}
      \psi\left(x+r\frac{x-y}{|x-y|}\right)(|x-y|+r)^{n-1}dr\right]\,dy; $$
  \item $$
    v(x)=\int_{x-\Omega}F(x-z)\;\frac{z}{|z|^{n}}\;\int_{0}^{\infty}\psi\left(x+r\frac{z}{|z|}\right)\;(|z|+r)^{n-1}\,dr\,dz,$$
    where $x-\Omega=\{z\in\R^{n}:\exists\, y\in\Omega\text{ such that } \;z=x-y\}$.
  \end{enumerate} 
\end{theorem}
Important consequences of the Theorem~\ref{thm:theorem1} are the next regularity results
for the ``potential'' $v$ in the interior, as well as at the boundary.
\begin{theorem}
  \label{thm:theorem2}
  Under the same hypotheses of Theorem~\ref{thm:theorem1}, if $F\in
  C_{0}^{\infty}(\Omega)$ then $v\in C_{0}^{\infty}(\Omega)$.  
\end{theorem}

\begin{theorem}
  \label{thm:theorem3}
  Under the same assumptions of Theorem~\ref{thm:theorem1}, it follows that\\ 
  $v\in C^{0}(\R^{n})\subseteq C^{0}(\overline{\Omega})$.
\end{theorem}

The proofs require several lemmas, and we start with an estimate for the kernel that
appears in the Bogovski\u\i's formula~\eqref{eq:bogovskij}.

\begin{definition}Let $\psi$ be as above. Then we set
\begin{equation*}
N(x,y):=\frac{x-y}{|x-y|^{n}}\int_{|x-y|}^{+\infty}\psi\left(y+\xi\frac{x-y}{|x-y|}\right)\xi^{n-1}d\xi,
\end{equation*}
and we remark that we can rewrite the Bogovski\u\i's formula as follows
\begin{equation*}
  v(x)=\int_\Omega N(x,y)\,F(y)\,dy.
\end{equation*}
\end{definition}
\begin{Lemma}
  \label{thm:lemma4}
There exists a constant $c>0$, depending only on  $n$, $\psi$, and $\diam\Omega$, such that
\begin{equation*}
|N(x,y)|\leq c|x-y|^{1-n}\quad \forall\, x,y\in\R^{n}:\;x\neq y.
\end{equation*}
\end{Lemma}
\begin{proof}
  Since the smooth function $\psi$ vanishes outside $B$, then it follows that
  $\psi\left(y+\xi\frac{x-y}{|x-y|}\right)=0$, when
  $\left|y+\xi\frac{x-y}{|x-y|}\right|\geq 1$. Since
  $\left|y+\xi\frac{x-y}{|x-y|}\right|\geq\left|\xi-|y|\,\right|$, we have that
  $\psi\big(y+\xi\frac{x-y}{|x-y|}\big)$ is zero for $\xi\in \R^+$ such that
  $\xi>1+\diam\,\Omega$, and therefore
  \begin{equation*}
    \left| \int_{|x-y|}^{+\infty}\psi\left(y+\xi\frac{x-y}{|x-y|}\right)\xi^{n-1}d\xi
    \right|\leq (1+\diam\,\Omega)^{n}\cdot\max_{x\in\R^{n}}|\psi(x)| =c,
  \end{equation*}
  and the lemma follows.
\end{proof}
The following simple remarks have important consequences in the study of the support of $v$.
\begin{Lemma}
  \label{thm:lemma5}
  If $x\notin\Omega$ and $\psi\left(y+\xi\frac{x-y}{|x-y|}\right)\neq 0$ holds true for
  some $\xi>|x-y|$, then $y\notin\Omega$.
\end{Lemma}
\begin{proof}
  In fact, since $\psi$ is not zero, if follows that
  \begin{equation*}
    \left|y+\xi\frac{x-y}{|x-y|}\right|\leq 1.
  \end{equation*}
  Moreover, $x=y+|x-y|\,\frac{x-y}{|x-y|}$ and hence $x$ belongs to the segment of
  endpoints $y$ and $y+\xi\frac{x-y}{|x-y|}$, for any $\xi>|x-y|$. If it were
  $y\in\Omega$, by the hypotheses on $\Omega$ the entire segment would lay in it, and that
  contradicts the assumption on $x$.
\end{proof}
An immediate consequence of the last lemma is the following result
\begin{Lemma}
\label{thm:lemma6}
The above function  $N(x,y)$ verifies
\begin{equation*}
  N(x,y)\;\equiv\;0\quad \forall\,x\notin\Omega,\text{ and }\forall\, y\in\Omega\ 
\end{equation*}
\end{Lemma}
We can now give the proof of Theorem~\ref{thm:theorem1}
\begin{proof}[Proof of Theorem~\ref{thm:theorem1}]
  By the above lemma, the vector $v(x)$ vanishes outside $\Omega$, and then \textit{ii)} follows from 
  \textit{i)}.
  \\
  To prove \textit{i)} and \textit{iii)}, fix $q>n$. Let $p$ be such that $1/p+1/q=1$, and
  then $p\in[1,\frac{n}{n-1}[$. Fix also any $x\in\Omega$. Since $\Omega$ is bounded, by
  Lemma~\ref{thm:lemma4} it turns out that $N(x,\cdot)\in L^{p}(\Omega)$ and by H\"{o}lder's inequality
  it follows that
\begin{equation*}
  \begin{aligned}
    |v(x)|&\leq\int_{\Omega}\left|F(y)\right|\,\left|N(x,y)\right|\,dy\leq
    c\int_{\Omega}\left|F(y)\right|\;|x-y|^{1-n}\,dy
    \\
    &\leq c \|F\|_{L^{q}(\Omega)}\left(\int_{\Omega}|x-y|^{p(1-n)}\,dy\right)^{1/p}.
  \end{aligned}
\end{equation*}
Since $\Omega\subseteq B(x,\,\diam\, \Omega$)
$$
\int_{\Omega}|x-y|^{p(1-n)}\,dy\leq \int_{B(x,\,\diam\,
  \Omega)}|x-y|^{p(1-n)}\,dy=\int_{B(0,\,\diam\, \Omega)}|z|^{p(1-n)}\,dz,
$$
and so \textit{ i)} and \textit{ii)} follow. 

Finally, by setting $g(z)=|z|^{1-n}$, we have 
$$
|v(x)|\leq c\| g\|_{L^{p}\left(B(0,\,\diam\,\Omega)\right)}\| F\|_{L^{q}(\Omega)},
$$
where the right side is independent of $x$. Since, by \textit{ ii)}, $v$ vanishes outside
$\Omega$, it follows immediately \textit{iii)}.

To get \textit{iv)}, it is enough to put $\xi=\alpha|x-y|$ in the initial formula. \\
By putting $r=\xi-|x-y|$, instead, it follows \textit{v)}. \\
Finally, by introducing $z=x-y$ in \textit{v)} it follows \textit{vi)}.
\end{proof}

\begin{remark}
  It is useful to remark explicitly that the Bogovski\u\i's ``potential'' $v$ vanishes at
  the boundary for any $F\in L^{q}(\Omega)$, without any other assumption than those made
  on $\Omega$ and $\psi$ in Theorem~\ref{thm:theorem1}. It is relevant to observe that it
  does not come by approximating $F$ by $C^\infty_0(\Omega)$ functions and by taking
  limits, but it is a property which descends directly from the formula for a large class
  of functions. In our case that is especially useful in that a given function in
  $C_D(\Omega)$ cannot be approximated in uniform norm by regular function with compact
  support, unless it vanishes at the boundary.
\end{remark}

\begin{proof}[Proof of Theorem~\ref{thm:theorem2}]
  From the formula in Theorem~\ref{thm:theorem1} \textit{vi)}, it follows that
\begin{equation*}
  \begin{aligned}
    v(x)&=\int_{x-\Omega}F(x-z)\;\frac{z}{|z|^{n}}\;\int_{0}^{\infty}\psi\left(x+r\frac{z}{|z|}\right)\;(|z|+r)^{n-1}\,dr\,dz
    \\
    &=\int_{x-\supp
      F}F(x-z)\;\frac{z}{|z|^{n}}\;\int_{0}^{1+\diam\,\Omega}\psi\left(x+r\frac{z}{|z|}\right)\;(|z|+r)^{n-1}\,dr\,dz.
  \end{aligned}
\end{equation*}
Since $\psi$ and $F$, as well as all their derivatives of any order, are bounded on
$\R^{n}$, the integrand is bounded by a multiple of the function $|z|^{1-n}$, which is integrable
on $\Omega$. By differentiating inside the integral, it follows $v\in C^{\infty}(\Omega)$.

Finally, in order to get $\supp v\subset \Omega$, let 
\begin{equation*}
E=\{z\in\Omega:\;z=(1-\lambda)z_{1} +\lambda z_{2}\quad z_{1}\in \supp F,\;
z_{2}\in\overline{B},\;\lambda\in[0,1]\},
\end{equation*}
which is a compact subset of $\Omega$ by the hypotheses, and fix any $x\in\Omega\backslash
E$. Now, it will be shown that $y+r(x-y)\notin\overline{B}$ for any $y\in \supp F$ and any
$r>1$. In fact, as it has been already seen in the proof of Lemma~\ref{thm:lemma5}, $x$ belongs to the
segment of endpoints $y$ and $y+r(x-y)$ for any $r\geq 1$. Thus, if
$y+r(x-y)\in\overline{B}$, it would follow $x\in E$, that contradicts the assumption on
$x$.
 
Hence, $\psi(y+r(x-y))\equiv 0\quad \forall x\in\Omega\backslash E$ and, by the formula
in Theorem~\ref{thm:theorem1} \textit{iv)}, it follows the theorem.
\end{proof}
\begin{proof}[Proof of Theorem~\ref{thm:theorem3}]
  Let $\{F_{k}\}\subset C^{\infty}_{0}(\Omega)$ such that $F_{ k}\rightarrow F$ in
  $L^{q}(\Omega)$ for some $q>n$, set to zero outside $\Omega$, and let $v_{k}$ and $v$
  the corresponding values obtained by the Bogovski\u\i's formula. By
  Theorem~\ref{thm:theorem1} \textit{ii)} and \textit{iii)}, it follows that $v_{k}$
  converge uniformly to $v$ in $\R^{n}$. Since, by Theorem~\ref{thm:theorem2},
  $\{v_{k}\}\subset C^{\infty}_{0}(\R^{n})$, the theorem follows immediately.
\end{proof}
\subsection{Further properties of the kernel}
In this section we prove some properties of the kernel appearing in the Bogovski\u\i's
formulas which turn to be useful in the following. Next lemma provides an identity about
the derivatives of the kernel as it appears in the second formula
(Theorem~\ref{thm:theorem1} \textit{iv)}).
\begin{Lemma} 
  \label{thm:lemma8}
  For any fixed $x,y\in \Omega,\ x\neq y$ let
  \begin{equation*}
    N_{i}(x,y):= (x_{i}-y_{i})\int_{1}^{\infty}\psi(y+\alpha(x-y))\,\alpha^{n-1}d\alpha.
  \end{equation*}
  Then, it follows that
  \begin{equation*}
    \partial_{x_{j}}N_{i}(x,y)=
    (x_{i}-y_{i})\int_{1}^{\infty}(\partial_{j}\psi)(y+\alpha(x-y))\,\alpha^{n-1}d\alpha
    -\partial_{y_{j}}N_{i}(x,y).
  \end{equation*}
\end{Lemma}
\begin{proof}
  Since $\left|y+\alpha(x-y)\right|\geq 1$ for $\alpha\geq(1+|y|)/|x-y|$ the integrand is
  bounded on a compact subset of $\R$. By differentiating under the sign of integral, it
  follows that
  \begin{equation*}
    \begin{aligned}
      \partial_{x_{j}}N_{i}(x,y)&=
      \delta_{ij}\;\int_{1}^{\infty}\psi(y+\alpha(x-y))\;\alpha^{n-1}d\alpha
      \\
      &\qquad +(x_{i}-y_{i})\int_{1}^{\infty}(\partial_{j}\psi)(y+\alpha(x-y))\alpha\;\alpha^{n-1}\,d\alpha,
    \end{aligned}
  \end{equation*}
  while
  \begin{equation*}
  \begin{aligned}
    \partial_{y_{j}}N_{i}(x,y)&=
    -\delta_{ij}\,\int_{1}^{\infty}\psi(y+\alpha(x-y))\;\alpha^{n-1}d\alpha
    \\
    &\qquad + (x_{i}-y_{i})\int_{1}^{\infty}(\partial_{j}\psi)(y+\alpha(x-y))\;(1-\alpha)\alpha^{n-1}\,d\alpha,
  \end{aligned}
\end{equation*}
and hence the lemma.
\end{proof}
The following result provides the fundamental estimates on $\partial_{x_{j}}N_{i}(x,y)$,
which allow to exploit the Calder\`{o}n-Zygmund theory to obtain the original Bogovski\u\i's
results about the $H_{0}^{1,p}(\Omega)$ regularity of $v$, and the Dini continuity
hypothesis to prove the results below.
\begin{theorem}
  \label{thm:theorem9}
  For any $i,j=1,\dots,n$ there exist functions $K_{ij}$ and $G_{ij}$ such that
$$\partial_{x_{j}}N_{i}(x,y)=K_{ij}(x,x-y)+G_{ij}(x,y)$$
where $K_{ij}(x,\cdot)$ is a Calder\`{o}n-Zygmund singular kernel and $G_{ij}$ is a weakly
singular kernel in the sense that, if one sets
\begin{equation*}
k_{ij}(x,z)\equiv |z|^{n}K_{ij}(x,z),
\end{equation*}
then, there exist constants $c=c(\psi,n)$ and $M=M(\psi,\,n,\,\diam\,\Omega)$ such
that:
\begin{enumerate}
\item $ \quad k_{ij}(x,tz)\,=\,k_{ij}(x,z)\quad\forall x\in\Omega,\;\forall z\neq 0,\; \forall t>0$;
\item $ \quad \| k_{ij}(x,z)\|_{L^{\infty}(\Omega\times S^{n-1})} $  is finite;
\item $ \quad \int_{|z|=1}k_{ij}(x,z)\,dz = 0\quad \forall x\in\Omega$;
\item $\quad |G_{ij}(x,y)|\leq c |x-y|^{1-n}$;     
\item $\quad |\partial_{x_{j}}N_{i}(x,y)| \leq M |x-y|^{-n}\quad\forall x\in\Omega\quad
  \forall y\in\R^{n}\backslash\{x\}$. 
\end{enumerate}
\end{theorem}
\begin{proof}
  Fix $x,y\in\Omega,\; x\neq y$. By introducing $r=\alpha|x-y|-|x-y| $ it follows that
  \begin{equation*}
    \begin{aligned}
      \partial_{x_{j}}&N_{i}(x,y)=\partial_{x_{j}}\left[
        (x_{i}-y_{i})\int_{1}^{\infty}\psi(y+\alpha(x-y))\alpha^{n-1}d\alpha \right]
      \\
      &=\delta_{ij}\int_{1}^{\infty}\psi(y+\alpha(x-y))\alpha^{n-1}d\alpha+
      (x_{i}-y_{i})\int_{1}^{\infty}\partial_{j}\psi(y+\alpha(x-y))\alpha^{n}d\alpha
      \\
      &=\frac{\delta_{ij}}{|x-y|^{n}}\int_{0}^{\infty}\psi\left(x+r\frac{x-y}{|x-y|}
  \right)(r+|x-y|)^{n-1}dr\;+
  \\
  &\quad+\frac{x_{i}-y_{i}}{|x-y|^{n+1}}\int_{0}^{\infty}\partial_{j}\psi\left(x+r\frac{x-y}{|x-y|}
  \right)(r+|x-y|)^{n}\,dr.
\end{aligned}
\end{equation*}
  By using the binomial expansion inside the integrals, the last sum may be written as follows
  $$
  K_{ij}(x,x-y)+G_{ij}(x,y),
  $$
  where
  \begin{equation*}
    \begin{aligned}
      K_{ij}(x,x-y)&=\frac{\delta_{ij}}{|x-y|^{n}}\int_{0}^{\infty}\psi\left(x+r\frac{x-y}{|x-y|}
      \right)r^{n-1}\,dr
      \\
      &\qquad
      +\frac{x_{i}-y_{i}}{|x-y|^{n+1}}\int_{0}^{\infty}\partial_{j}\psi\left(x+r\frac{x-y}{|x-y|}
      \right)r^{n}\,dr,
\end{aligned}
\end{equation*}
involves only the terms not containing any strictly positive power of $|x-y|$, while all
of the others, grouped as $G$, contain at least a factor $|x-y|$ coming from the
expansion; therefore, $G_{ij}$ verifies the estimate in \textit{ iv)}.

The homogeneity property in \textit{i)} follows immediately from the previous expression
of $K_{ij}$.

Furthermore, $z\in S^{n-1}$ implies
\begin{equation*}
  \begin{aligned}
    |k_{ij}(x,z)|&\leq\left| \int_{0}^{\infty}\psi\left(x+r\,z
      \right)r^{n-1}dr\right|+ \left| \int_{0}^{\infty}\partial_{j}\psi\left(x+r\,z
      \right)r^{n}dr\right|
    \\
    &
    \leq\|\psi\|_{L^{\infty}(\R^{n})}(\diam\,\Omega)^{n}+\| \partial_{j}\psi\|_{L^{\infty}(\R^{n})}
    (\diam\,\Omega)^{n+1},
  \end{aligned}
\end{equation*}
and \textit{ii)} follows. 

Finally,
\begin{equation*}
  \begin{aligned}
    \int_{|z|=1}&k_{ij}(x,z)\,dz=
    \\
    & = \delta_{ij}\int_{|z|=1}\int_{0}^{\infty}\psi\left(x+rz \right)r^{n-1}dr+
    \int_{|z|=1}\;z_{i}\int_{0}^{\infty}\partial_{j}\psi\left(x+rz \right)r^{n}\,dr
    \\
    &=\int_{\R^{n}}\left[ \delta_{ij}\psi(x+y)+y_{i}\partial_{j}\psi(x+y)\right]\,dy.
  \end{aligned}
\end{equation*}
After an integration by parts, the last integral turns out to be zero, and \textit{
  iii)} is proved.

Finally, \textit{ii)} and \textit{iv)} imply immediately \textit{v)}, on the bounded set
  $\Omega$.
\end{proof}
\subsection{The approximating functions for the solution}
The main tool we applied in this paper is an old aged argument exploited by Korn (see,
e.g., Gilbarg and Trudinger~\cite[Ch. 4]{GT1998}) in the study of the existence of
classical solutions of the Poisson equation, based on a suitable ''cutoff'' of the
singularity present in the second derivatives of the Newtonian potential, that provides a
way to approximate the solution $v$ by regular functions.

To this end we introduce the function $\eta\in C^{\infty}(\R^{+})$ such that
$\eta(t)\equiv 0$ on $[0,1]$, $\eta(t)\equiv 1$ if $t\geq 2$, and $|\eta'(t)|\leq 2\quad
\forall t \in \R^{+}$.
\begin{definition}
For any $q>n$, $F\in L^{q}(\Omega)$ and $\epsilon>0$ let us set
\begin{equation*}
  \begin{aligned}
    v^{\epsilon}(x)&=\int_{\Omega}F(y)\;(x-y)\;\left[
      \int_{1}^{\infty}\psi(y+\alpha(x-y))\;\alpha^{n-1}d\alpha\right]\;\eta\left(\frac{|x-y|}{\epsilon}\right)\,dy
\\
&=\int_{\Omega}F(y)\;N(x,y)\;\eta\left(\frac{|x-y|}{\epsilon}\right)\,dy.
\end{aligned}
\end{equation*}
\end{definition}
We remark that if $|x-y|<\epsilon$ the integrand is zero and therefore the integrand
belongs to $ C^{\infty}(\R^{n})$ and it is bounded by Lemma~\ref{thm:lemma4}, while if\\
$|x-y|\geq\epsilon$ the set of $\alpha$ such that $|y+\alpha(x-y)|\leq 1$, where $\psi$
could be not null, is bounded as well. It follows immediately that $v^{\epsilon}$ is
well-defined and belongs to $C^{\infty}(\R^{n})$.  Moreover, under the same hypotheses of
Theorem~\ref{thm:theorem2} and following the same argument, one proves that it belongs
actually to $C^{\infty}_{0}(\Omega)$.
\begin{theorem}
  Assume the same hypotheses of Theorem~\ref{thm:theorem1}, and let $\eta$ and
  $v^{\epsilon}$ as above.  Then:
\begin{enumerate}
\item If, in addition, $F\in L^{\infty}(\Omega)$, then
$$
v^{\epsilon} \rightarrow v  \;\text{uniformly for }x\in\R^{n};
$$
\item If, moreover, $\int_{\R^{n}}\psi(x)\,dx=1$ and $F\in C^{0}(\overline{\Omega})$, then
\begin{equation*}
\lim_{\epsilon\to 0}\,\dive v^{\epsilon}(x)=-\psi(x)\int_{\Omega} F(y)\,dy+F(x) \qquad\forall\, x\in\Omega.
\end{equation*}
\end{enumerate}
\end{theorem}
\begin{proof}
  To prove \textit{i)}, let us fix any $x\in\Omega$. Remark that, by the
  Lemma~\ref{thm:lemma4}, for any $\epsilon<\textrm{dist}(x,\partial \Omega)$
\begin{equation*}
  \begin{aligned}
    |v^{\epsilon}_{i}(x)-v_{i}(x)|&\leq
    \int_{\Omega}\left|\;F(y)N_{i}(x,y)\right|\;\left|\eta\left(\frac{|x-y|}{\epsilon}\right)\;-1\;\right|\,dy
    \\
    &\leq c \| F\|_{L^{\infty}(\Omega)}\int_{|x-y|<\epsilon}\frac{1}{|x-y|^{n-1}}\,dy
    \leq c \| F\||_{L^{\infty}(\Omega)}\int_{|z|<\epsilon}|z|^{1-n}\,dz,
  \end{aligned}
\end{equation*}
and by the absolute continuity of the last integral, it follows than it tends to zero
independently of $x\in\Omega$. To complete the proof of
\textit{i)}  it is enough to remark that by Lemma~\ref{thm:lemma6} $v^{\epsilon}(x)=v(x)= 0\quad \forall x\notin\Omega$. \\

To prove \textit {ii)}, by differentiating $v^{\epsilon}_{i}$ at any $x\in\Omega$ it
follows that
\begin{equation*}
  \begin{aligned}
    \partial_{x_{j}} &v^{\epsilon}_{i}(x)=\int_{\Omega}F(y)\delta_{ij}\,\left[
      \int_{1}^{\infty}\psi(y+\alpha(x-y))\;\alpha^{n-1}d\alpha\right]\,dy
    \\
    &+\int_{\Omega}F(y)(x_{i}-y_{i})\;\left[\int_{1}^{\infty}
      (\partial_{j}\psi(y+\alpha(x-y))\;\alpha^{n}d\alpha\right]\,\eta\left(\frac{|x-y|}{\epsilon}\right)\,dy
    \\
    &+\int_{\Omega}F(y)(x_{i}-y_{i})\,\left[\int_{1}^{\infty}\psi(y+\alpha(x-y))\;\alpha^{n-1}d\alpha\right]\,\eta'\left(\frac{|x-y|}{\epsilon}\right)\frac{x_{j}-y_{j}}{|x-y|}\,\epsilon^{-1}\,dy
    \\
    &=
    \int_{\Omega}F(y)\delta_{ij}\;\left[\int_{1}^{\infty}\psi(y+\alpha(x-y))\;\alpha^{n-1}d\alpha\right]\;\eta\left(\frac{|x-y|}{\epsilon}\right)\,dy
    \\
    &+\int_{\Omega}F(y)(x_{i}-y_{i})\,\left[\int_{1}^{\infty}
      (\partial_{j}\psi(y+\alpha(x-y))\;\alpha^{n}d\alpha\right]\,\eta\left(\frac{|x-y|}{\epsilon}\right)\,dy
    \\
    &+\int_{\Omega}F(y)\frac{(x_{i}-y_{i})(x_{j}-y_{j})}{|x-y|}\;\left[\int_{1}^{\infty}\psi(y+\alpha(x-y))\;\alpha^{n-1}d\alpha\right]\;\eta'\left(\frac{|x-y|}{\epsilon}\right)\;\epsilon^{-1}dy,
  \end{aligned}
\end{equation*}
and therefore
\begin{equation*}
  \begin{aligned}
    \dive
    v^{\epsilon}(x)&=n\;\int_{\Omega}F(y)\;\left[\int_{1}^{\infty}\psi(y+\alpha(x-y))\;\alpha^{n-1}d\alpha\right]\;\eta\left(\frac{|x-y|}{\epsilon}\right)\,dy
    \\
    &+\sum_{i=1}^{n}\int_{\Omega}F(y)(x_{i}-y_{i})\;\left[\int_{1}^{\infty}
      (\partial_{j}\psi(y+\alpha(x-y))\;\alpha^{n}d\alpha\right]\;\eta\left(\frac{|x-y|}{\epsilon}\right)\,dy
    \\
    &+\int_{\Omega}F(y)|x-y|\;\left[\int_{1}^{\infty}\psi(y+\alpha(x-y))\;\alpha^{n-1}d\alpha\right]\;\eta'\left(\frac{|x-y|}{\epsilon}\right)\;\epsilon^{-1}dy
    \\
    &=:A+B+C.
\end{aligned}
\end{equation*}
Now,
\begin{equation*}
  \begin{aligned}
    A+B&= 
    \int_{\Omega}F(y)\eta\left(\frac{|x-y|}{\epsilon}\right)\quad\times
\\
&\times\;\int_{1}^{\infty}\left[\psi(y+\alpha(x-y))\,n\alpha^{n-1}+\alpha^{n}\sum_{i=1}^{n}\partial_{x_{i}}\psi(y+\alpha(x-y))(x_{i}-y_{i})\right]
\,d\alpha\,dy
\\
&=\quad \int_{\Omega}F(y)\;\eta\left(\frac{|x-y|}{\epsilon}\right) \int_{1}^{\infty}
\frac{d}{d\alpha} \left[\psi(y+\alpha(x-y))\alpha^{n}\right]\,d\alpha\,dy
\\
&=\quad -\psi(x)\;\int_{\Omega}F(y)\;\eta\left(\frac{|x-y|}{\epsilon}\right)\,dy\quad
\longrightarrow\quad -\psi(x)\int_{\Omega}F(y)\,dy,
\end{aligned}
\end{equation*}
as $\epsilon$ tends to zero.  Moreover, by setting into $C$ first $\alpha=\xi/|x-y|$, next
$\xi=r+|x-y|$, and finally $z=\epsilon^{-1}(x-y)$ one obtains
\begin{equation*}
  \begin{aligned}
    C&=
    \int_{\Omega}\frac{F(y)}{|x-y|^{n-1}}\;\eta'\left(\frac{|x-y|}{\epsilon}\right)\;\epsilon^{-1}\;\left[\int_{|x-y|}^{\infty}\psi\left(y+\xi\frac{x-y}{|x-y|}\right)\xi^{n-1}\,d\xi\right]\,dy
\\
&
=
\int_{\epsilon<|x-y|<2\epsilon}\frac{F(y)}{|x-y|^{n-1}}\;\eta'\left(\frac{|x-y|}{\epsilon}\right)\;\epsilon^{-1}\;\times
\\
&\qquad \qquad\times\;\left[\int_{0}^{\infty}\psi\left(x+r\frac{x-y}{|x-y|}\right)
  (r+|x-y|)^{n-1}\,dr\right]\,dy 
\\
&= \int_{1<|z|<2}\frac{F(x-\epsilon z)}{|z|^{n-1}}\;\eta'(|z|)\; \left[
  \int_{0}^{\infty}\psi\left(x+r\frac{z}{|z|}\right) (r+\epsilon|z|)^{n-1}\,dr\right]\,dz.
\end{aligned}
\end{equation*}
Now we claim that, as $\epsilon$ goes to $0$, the last term tends to
\begin{equation*}
F(x)\int_{1<|z|<2}\frac{1}{|z|^{n-1}}\eta'(|z|)\;\left[
  \int_{0}^{\infty}\psi\left(x+r\frac{z}{|z|}\right) r^{n-1}\,dr\right]\,dz,
\quad\quad\quad (\bigstar)
\end{equation*}
In fact, since $\psi\left(x+r\frac{z}{|z|}\right)$ vanishes when $r>1+\diam\,\Omega$, the
inner integral is bounded by $\max |\psi|\;(1+\diam\,\Omega)^{n}=:M$. Hence, we get
\begin{equation*}
  \begin{aligned}
    &\left| \int_{1<|z|<2}\frac{F(x-\epsilon \,z)-F(x)}{|z|^{n-1}}\;\eta'(|z|)\; \left[
        \int_{0}^{\infty}\psi\left(x+r\frac{z}{|z|}\right)
        (r+\epsilon|z|)^{n-1}\,dr\right]\,dz\right|
    \\
    &\quad \leq 2M \int_{1<|z|<2}\frac{|F(x-\epsilon \,z)-F(x)|}{|z|^{n-1}}\,dz\leq 2M
    \int_{1<|z|<2}\frac{\omega(\epsilon\,|z|)}{|z|^{n-1}}\,dz,
\end{aligned}
\end{equation*}
where $\omega$ is the modulus of continuity of $F$. By the uniform continuity of $F$ on
$\Omega$ and the Lebesgue theorem on dominated convergence, the last integral vanishes as
$\epsilon$ goes to zero and therefore the claim is proved.

Finally, by introducing in the limit $(\bigstar)$ the radial and angular coordinates $
\rho=|z|$ and $u=z/|z|$, one gets
\begin{equation*}
  \int_{1}^{2}\eta'(\rho)\,d\rho\int_{S^{n-1}}\int_{0}^{\infty}\psi(x+ru)\,r^{n-1}\,dr\,du=\left(\eta(2)-\eta(1)\right)
  \int_{\R^{n}}\psi(w)\,dw=1.
\end{equation*}
Therefore, $C\longrightarrow F(x)$ as $\epsilon$ tends to $0$ and the lemma follows.
\end{proof}
%
%
%
%
The following theorem, an immediate corollary of the previous one, is the cornerstone of
the resolution of the divergence problem.
\begin{theorem}
  \label{thm:theorem11}
  Assume the same hypotheses of the previous theorem.  Moreover, let
  $\int_{\R^{n}}\psi(y)\,dy=1$, $F\in C^{0}(\overline{\Omega})$ and $\int_{\Omega}F(x)\,dx=0$.
  Then, as $\epsilon$ goes to zero,
\begin{equation*}
\dive  v^{\epsilon}(x)\rightarrow F(x)\qquad\forall  x\in\Omega.
\end{equation*}
\end{theorem}
%
The next lemma, an immediate consequence of Lemma~\ref{thm:lemma8} and the opposite sign
in the derivatives of $\eta(|x-y|/\epsilon)$, will be useful in proving the subsequent
representation formula. 
\begin{Lemma}
  \label{thm:lemma12}
\begin{equation*}
  \begin{aligned}
    &\partial_{x_{j}}\left[N_{i}(x,y)\,\eta\left(\frac{|x-y|}{\epsilon}\right)\right]=- \partial_{y_{j}}\left[N_{i}(x,y)\,\eta\left(\frac{|x-y|}{\epsilon}\right)\right]
    +
    \\
    &\qquad \qquad
    +\eta\left(\frac{|x-y|}{\epsilon}\right)(x_{i}-y_{i})\int_{1}^{\infty}\partial_{j}\psi(y+\alpha(x-y))\;\alpha^{n-1}d\alpha.
\end{aligned}
\end{equation*}
\end{Lemma}
%
%
As usual in potential theory, getting a representation formula for the derivatives of the
function $v_{i}$ is a crucial goal. We will obtain it through a limit of the derivatives
of its ``regular approximation'' $v^{\epsilon}$. Thus, let us start by differentiating the
formula
\begin{equation*}
  v_{i}^{\epsilon}(x)= \int_{\Omega}
  F(y)N_{i}(x,y)\,\eta\left(\frac{|x-y|}{\epsilon}\right)\,dy= \int_{B_{R}}
  F(y)N_{i}(x,y)\,\eta\left(\frac{|x-y|}{\epsilon}\right)\,dy, 
\end{equation*}
where $F\in L^{\infty}(\Omega)$ is extended by zero outside $\Omega$ and $B_{R}$ is a ball
of radius large enough such that $\Omega \subset\subset B_{R}$.  By the previous lemma and
Bogovski\u\i's formula in Theorem~\ref{thm:theorem1} \textit{iv)} it follows that
\begin{equation*}
  \begin{aligned}
    \partial_{x_{j}}v_{i}^{\epsilon}(x)&=\int_{B_{R}}
    F(y)\;\partial_{x_{j}}\left[N_{i}(x,y)\eta\left(\frac{|x-y|}{\epsilon}\right)\right]\,dy
    \\
    &=\int_{B_{R}}
    \left[F(y)-F(x)\right]\partial_{x_{j}}\left[N_{i}(x,y)\eta\left(\frac{|x-y|}{\epsilon}\right)\right]\,dy
    \\
    &\qquad+F(x)\int_{B_{R}}\partial_{x_{j}}\left[N_{i}(x,y)\eta\left(\frac{|x-y|}{\epsilon}\right)\right]\,dy
    \\
    &=\int_{B_{R}}
    \left[F(y)-F(x)\right]\partial_{x_{j}}\left[N_{i}(x,y)\eta\left(\frac{|x-y|}{\epsilon}\right)\right]\,dy
    \\
    &\qquad+F(x)\;\int_{B_{R}}\eta\left(\frac{|x-y|}{\epsilon}\right)(x_{i}-y_{i})\;\int_{1}^{\infty}\partial_{j}\psi(y+\alpha(x-y))\;\alpha^{n-1}d\alpha\,dy 
    \\
    &\qquad
    -F(x)\int_{B_{R}}\partial_{y_{j}}\left[N_{i}(x,y)\eta\left(\frac{|x-y|}{\epsilon}\right)\right]\,dy.
\end{aligned}
\end{equation*}
Since $\eta\left(\frac{|x-y|}{\epsilon}\right)=1$ if $\epsilon<\textrm{dist}(\partial
B_{R},\overline{\Omega})$, by the Gauss-Green formula the last integral is equal to
$\int_{\partial B_{R}}N_{i}(x,y)\nu_{j}(y)d\sigma_{y}$.

\medskip

The previous computation suggests to put forward a conjecture about the limit as
$\epsilon$ goes to zero, which will be proved in the next theorem, that is the main result
of the paper.
%
%
\begin{theorem} 
  \label{thm:theorem13}
  Assume all the hypotheses of the Theorem~\ref{thm:theorem1}, and let $\eta$ and
  $v^{\epsilon}$ as above.  Furthermore, let $\psi$ be such that $\int_{\R^{n}}\psi=1$,
  $F\in C_{D}(\Omega)$, and
\begin{equation*}
  \begin{aligned}
    u_{j}^{i}(x)=&\int_{B_{R}}
    \left[F(y)-F(x)\right]\partial_{x_{j}}N_{i}(x,y)\,dy
    \\
    &+F(x)\int_{B_{R}}(x_{i}-y_{i})\int_{1}^{\infty}\partial_{j}\psi(y+\alpha(x-y))\;\alpha^{n-1}d\alpha\,dy\;-
    \\
    &-F(x)\;\int_{\partial B_{R}}N_{i}(x,y)\nu_{j}(y)d\sigma_{y}.
  \end{aligned}
\end{equation*}
Then:
\begin{enumerate}
\item\quad	$u_{j}^{i}(x)$ is well-defined for any $x\in\Omega$;
\item\quad $\partial_{x_{j}}v^{\epsilon}_{i} $ converges uniformly to $u_{j}^{i} $ on any
  $\Omega'\subset\subset\Omega$;
\item\quad $\partial_{x_{j}}v\;\equiv\;u_{j}^{i} \quad on\; \Omega$;
\item\quad   $v\in C^{1}(\Omega)$.
\end{enumerate}
\end{theorem}
\begin{proof}
  To prove \textit{i)}, fix any $x\in\Omega$. Remark that, after its extension by zero
  outside $\Omega$, $F\in L^{\infty}(\R^{n})$. For any
  $\epsilon<\textrm{dist}(x,\partial\,\Omega)$ one has
\begin{equation*}
  \begin{aligned}
    &\int_{B_{R}} \left|F(y)-F(x)\right|\left|\partial_{x_{j}}N_{i}(x,y)\right|\,dy
    \\
    &=\int_{B(x,\epsilon)}\left|F(y)-F(x)\right|\left|\partial_{x_{j}}N_{i}(x,y)\right|\,dy
    \\
    &\qquad +\int_{\{|x-y|\geq \epsilon\}\cap
      B_{R}}\left|F(y)-F(x)\right|\left|\partial_{x_{j}}N_{i}(x,y)\right|\,dy
    \\
    &=:A+C.
\end{aligned}
\end{equation*}
Since $B(x,\epsilon)\subset\Omega$, by Theorem~\ref{thm:theorem9} \textit{v)} it follows that
\begin{equation*}
  \begin{aligned}
    A&\leq\int_{B(x,\epsilon)}
    \frac{\left|F(y)-F(x)\right|}{|y-x|}|y-x|\;|\partial_{x_{j}}N_{i}(x,y)|\,dy
    \\    
&\leq\int_{B(x,\epsilon)} \frac{\omega(|y-x|)}{|y-x|}\;\frac{M}{|y-x|^{n-1}}\,dy,
\end{aligned}
\end{equation*}
where $\omega$ is the modulus of continuity of $F$ in $\Omega$. By introducing the radial
and angular coordinates, the last integral becomes
\begin{equation*}
M|S^{n-1} | \int_{0}^{\epsilon} \frac{\omega(F,\rho)}{\rho}\,d\rho,
\end{equation*}
and, by the Dini continuity hypothesis on $F$, it is finite.

Furthermore, since both $F$ and $\partial_{x_{j}}N_{i}(x,y)$ are bounded on $\{|x-y|\geq
\epsilon\}$, the term  C is finite as well.

Finally, since $\partial_{j}\psi\in C^{\infty}_{0}(\R^{n})$ and $\supp \partial_{j}\psi
\subset B$, it follows that
\begin{equation*}
\int_{\Omega}(x_{i}-y_{i})\int_{1}^{\infty}\partial_{j}\psi(y+\alpha(x-y))\;\alpha^{n-1}d\alpha\,dy,
\end{equation*}
is the value of Bogovski\u\i's formula corresponding to the bounded function $F\equiv 1$,
evaluated by using $\partial_{j}\psi$ instead of $\psi$. By Theorem~\ref{thm:theorem1}
\textit{iii)}, it is globally bounded, and \textit{i)} follows.

To prove \textit{ii)}, fix any $\Omega'\subset\subset\Omega$. Thus, for any $x\in\Omega'$
and $\epsilon>0$ such that $2\epsilon<\textrm{dist}(\overline{\Omega'},\partial \Omega)$, it follows
that
\begin{equation*}
  \begin{aligned}
    |\partial_{x_{j}}&v^{\epsilon}_{i}(x) -u_{j}^{i}(x)|\leq
    \\
    &\leq \left|\int_{B_{R}}
      \left[F(y)-F(x)\right]\partial_{x_{j}}\left\{N_{i}(x,y)\;\left[\eta\left(\frac{|x-y|}{\epsilon}\right)-1\right]\right\}\,dy\;\right|
    \\
    &+\left|F(x)\int_{B_{R}}(x_{i}-y_{i})\int_{1}^{\infty}\partial_{x_{j}}\psi(y+\alpha(x-y))\;\left[\eta\left(\frac{|x-y|}{\epsilon}\right)-1\right]\;\alpha^{n-1}d\alpha\,dy\;\right|
    \\
    &\leq \int_{B(x,2\epsilon)} |F(x)-F(y)|\;|\partial_{x_{j}}N_{i}(x,y)|\,dy+
    \\
    &+\int_{B(x,2\epsilon)}
    |F(y)-F(x)|\;|N_{i}(x,y)|\;\left|\eta'\left(\frac{|x-y|}{\epsilon}\right)\frac{x_{j}-y_{j}}{|x-y|}\;\epsilon^{-1}\right|\,dy
    \\
    &+\int_{B(x,2\epsilon)}|F(x)|\;|x_{i}-y_{i}|\int_{1}^{\infty}|\partial_{x_{j}}\psi(y+\alpha(x-y))\;|\;\alpha^{n-1}d\alpha\,dy\;|
    \\
    &=:D+E+H.
\end{aligned}
\end{equation*}
As above, by Theorem~\ref{thm:theorem9} \textit{v)} it follows that
\begin{equation*}
D\leq M\;\int_{ B(x,2\epsilon)}\frac{
  |F(x)-F(y)|}{|y-x|^{n}}\,dy\leq M\;|S^{n-1}|\;\int_{\rho<2\epsilon}\frac{\omega(F,\rho)}{\rho}\,d\rho.
\end{equation*}
By the Dini continuity of $F$ and the consequent absolute continuity of the integral, the
last term vanishes as $\epsilon$ goes to zero, independently of $x\in\Omega'$.

In order to estimate the second term $E$ remark that, by Theorem~\ref{thm:theorem1}
\textit{v)} and the hypothesis on $\eta'$
\begin{equation*}
  \begin{aligned}
    E&\leq \int_{\epsilon\leq |x-y|\leq 2\epsilon}
    |F(x)-F(y)| \left|\eta'\left(\frac{|x-y|}{\epsilon}\right)\right|\frac{|x_{j}-y_{j}|}{|x-y|} \epsilon^{-1} \times
    \\
    & \times  \frac{|x_{i}-y_{i}|}{|x-y|^{n}}\int_{0}^{\infty}
    \left|\psi\left(x+r\frac{x-y}{|x-y|}\right)\right|(|x-y|+r)^{n-1}dr\,dy
    \\
    & \leq 4  \int_{\epsilon\leq |x-y|\leq 2\epsilon}
    \left|F(x)-F(y)\right| \frac{1}{|x-y|^{n}} \times
    \\
    & \qquad \times \int_{0}^{\infty}
    \left|\psi\left(x+r\frac{x-y}{|x-y|}\right)\right|(|x-y|+r)^{n-1}dr\,dy.
  \end{aligned}
\end{equation*}
By introducing the variable $y=x+\rho u$, since
\begin{equation*}
  \begin{aligned}
    \int_{0}^{\infty}|\psi(x+ru)|(\rho+r)^{n-1}\,dr&=\int_{0}^{1+|x|}
    \left|\psi(x+ru)\right|(\rho+r)^{n-1}\,dr
\\
&\leq \max_{\R^{n}}|\psi|\;(1+\diam\,\Omega+2\epsilon)^{n-1},
\end{aligned}
\end{equation*}
it follows as above that the last term is bounded by a multiple of
$\int_{\epsilon}^{2\epsilon}\frac{\omega(F,\rho)}{\rho}\,d\rho$ and, again by the absolute
continuity of the integral, $E$ vanishes as $\epsilon$ goes to zero, independently of
$x\in\Omega'$.

Finally, 
by using $\partial_{j}\psi$ instead of $\psi$ as in the proof of the previous
\textit{i)}, from Theorem~\ref{thm:theorem1} \textit{iii)} it follows that for any $q\,>n$ and suitable
constants $c', c''$
$$
|H|\,\leq \,c'\,\max_{\overline{\Omega}}|F(x)|\,\left\|
  \eta\left(\frac{|x-y|}{\epsilon}\right)-1\right\|_{L^{q}(\Omega)}\,\leq\,
$$
$$\leq c''\,
\left\|
  \eta\left(\frac{|x-y|}{\epsilon}\right)-1\right\|_{L^{q}\left(B(x,\,\diam\,\Omega)\right)}.
$$
Since the last norm vanishes as $\epsilon$ goes to zero, for any $q\,>\,n$ and independently of $x\in\Omega$, \textit{ii)} follows.\\

From \textit{ii)}, by the classical theorem on a convergent sequence of functions whose
derivatives converge uniformly, it follows \textit{iii)}, while \textit{ iv)} follows
immediately from \textit{ ii)}, \textit{iii)}, since $v^{\epsilon}\in C^{\infty}(\R^{n})$.
\end{proof}
All the previous results lead to the following theorem.

\begin{theorem}
  \label{thm:theorem14}
  Let $B=B(x_{0},R)$ be an open ball in $\R^{n}, n\geq 2$, and let $\psi$ be any function
  in $C_{0}^{\infty}(\R^{n})$, verifying $\supp \; \psi\subseteq B$ and
  $\int_{\R^{n}}\psi=1$.  Let $\Omega$ be a bounded open subset of $\R^{n}$, star-shaped
  with respect to every point of $\overline{B}$. Then, for any $F\in C_{D}(\Omega)$
  verifying $\int_{\Omega}F=0$, the Bogovski\u\i's formula
\begin{equation*}
  v(x)=\int_{\Omega}F(y)\left[\frac{x-y}{|x-y|^{n}}\int_{|x-y|}^{+\infty}\psi\left(y+\xi\frac{x-y}{|x-y|}\right)\xi^{n-1}d\xi\right]\,dy,
\end{equation*}
defines a solution $v\in C^{1}(\Omega)\cap C^{0}(\R^{n})$ of the problem
\begin{equation*}
\left\{
\begin{array}{ll}
  \dive v(x)\,=\,F(x)& \text{in }\Omega,
  \\
  v \equiv 0 & \text{on }\complement\Omega.
\end{array}
\right.
\end{equation*}
\end{theorem}
\begin{proof}
  At first observe that, if $B=B(0,1)$, the theorem follows immediately from
  Theorem~\ref{thm:theorem1} \textit{ii)}, Theorem~\ref{thm:theorem3},
  Theorem~\ref{thm:theorem11} and Theorem~\ref{thm:theorem13}.

  Otherwise, let us set $z=(x-x_{0})/R$, $\widetilde{\Omega}=\{(x-x_{0})/R:x\in\Omega \}$,
  $\widetilde{F}(z)=F(x_{0}+Rz)$ and remark that $\widetilde{\Omega}$ and $\widetilde{F}$ fulfil the
  previous hypotheses with respect to $B=B(0,1)$.  Now, let $w(z)$ be the solution of
  \begin{equation*}
\left\{
  \begin{array}{ll}
    \dive w(z)\,=\,\widetilde{F}(z)& \text{in } \widetilde{\Omega},
    \\
    w \equiv 0 & \text{on }  \complement\widetilde{\Omega},
  \end{array}
\right.
\end{equation*}
whose existence follows by the initial observation, and remark that
\begin{equation*}
  \dive\left[R\;w\left(\frac{x-x_{0}}{R}\right)\right]=(\dive
  w)\left(\frac{x-x_{0}}{R}\right)=
\widetilde{F}\left(\frac{x-x_{0}}{R}\right)=F(x).
\end{equation*}
Therefore
\begin{equation*}
  v(x)=R\;w\left(\frac{x-x_{0}}{R}\right),
\end{equation*}
is the requested solution, and the proof is completed.
\end{proof}
\section{Existence of classical solutions for the divergence problem in more general
  domains.}
The aim of this brief section is to relax the very strong geometric restriction on the
domain $\Omega$ requested in the above result, although at the price to renounce the
simplicity of a single Bogovski\u\i's representation formula for the solution of the
divergence problem. 

The next theorem provides the existence of a classical solution in a wider class of
domains including, for instance, those with a smooth boundary. To this aim, we start to
prove a suitable ``partition of unity'' lemma.
\begin{lemma}
  \label{thm:partition}
  Let $\Omega$ be a bounded subset of $\R^{n}$ with a locally Lipschitz boundary.  Then,
  there exists an open covering $\mathcal{G}=\{G_{1},\dots, G_{m}, G_{m+1},\,\dots\,,
  G_{m+p}\}$ of $\overline{\Omega}$ such that, if one sets $\Omega_{i}:= \Omega\cap G_{i}
  $ it follows that:
  \begin{itemize}
  \item
    $\Omega_{i}$ is star-shaped with respect to every point of an open ball $B_{i}$, with
    $\overline{B}\subset\Omega$ for any $i=1,\dots,m+p$; 
  \item
    $\partial\Omega\subset\cup_{1}^{m}G_{i};$
  \item
    $G_{i}$ is an open ball with closure in $\Omega$ for any $i=m+1,\dots,m+p$;
  \item
    $\Omega=\cup_{1}^{m+p}\Omega_{i}$.
  \end{itemize}
  Furthermore, for any fixed $F\in C_{D}(\Omega)$ with $\int_{\Omega}F=0$, there exist $F_{i}\in C_{D}(\Omega)$ such that:
  \begin{enumerate}
  \item 
    $F_{i}\;\equiv 0\quad on\;\Omega\backslash \Omega_{i}$ for any $i=1,\dots,m+p$
  \item
    $F\;\equiv\;\sum_{1}^{m+p}F_{i}$  on $\Omega$
  \item
    $\int_{\Omega}F_{i}=0$ for any $i=1,\dots,m+p$
  \end{enumerate}
  \begin{proof}
    The proof if this result may be obtained as in Galdi~\cite[Lemma~III.3.4]{Gal2011}, by
    replacing $C_{0}^{\infty}$ with $C_{D}$ in any occurrence involving $f$, $f_{i}$ or
    $g_{i}$, by assuming $\Omega$ as their domain, and by extending $\psi_{i}$ and
    $\chi_{i}$ by zero outside their supports.
  \end{proof}
\end{lemma}
\begin{remark}
  We remark explicitly that from \textit{i)} and \textit{iii)} it follows immediately the
  crucial property
\begin{equation*}
 \int_{\Omega_{i}}F_{i}(x)\,dx=0,
\end{equation*}
which, together with the properties of the covering, allows to apply the regularity result
in Theorem~\ref{thm:theorem14} to the divergence problem ''localized'' at $\Omega_{k}$.
\end{remark}
The final result, which extends Theorem~\ref{thm:theorem14} to a considerably wider class
of domains, will be now obtained by a localization argument.
\begin{theorem}
  Let $\Omega$ be a bounded subset of $\R^{n}$ with a locally Lipschitz boundary.  Then,
  for any $F\in C_{D}(\Omega)$ with $\int_{\Omega}F(x)\,dx=0$ there exists a solution
  $v\in C^{1}(\Omega)\cap C^{0}(\overline{\Omega}) $ of the problem
  \begin{equation*}
    \left\{
      \begin{aligned}
        \dive v(x)&=F(x)\qquad &\text{in } \Omega,
        \\
        v &\equiv 0 \qquad&\text{on } \partial\Omega.
      \end{aligned}
    \right.
  \end{equation*}
\end{theorem}

\begin{proof}
  Let $\Omega_{k}$ and $F_{k}$ be defined as in the previous lemma, and let $v_{k}$ be the
  solution in $C^{1}(\Omega_{k})\cap C^{0}(\R^{n})$ of the problem
  \begin{equation*}
\left\{
  \begin{aligned}
    \dive v_{k}(x)&=F_{k}(x) \qquad  &\text{in } \Omega_{k},
    \\
    v_{k} &\equiv 0\qquad  &\text{on }\complement\Omega_{k},
  \end{aligned}
\right.
\end{equation*}
whose existence is ensured by Theorem~\ref{thm:theorem14}, Lemma~\ref{thm:partition}, and
the last remark.  Thus, by setting
 \begin{equation*}
v(x)=\sum_{1}^{m+p}v_{k}(x),
\end{equation*}
one obtains $v\in C^{0}(\R^{n})$. Moreover, since $v_{k}$ vanishes on
$\complement\Omega_{k}$ then $v\equiv 0$ on $\partial\Omega$ and
 \begin{equation*}
v(x)=\sum_{k:\;\Omega_{k}\ni x}v_{k}(x),
\end{equation*}
and hence $v\in C^{1}(\Omega)$. 

Finally, for any $x\in\Omega$
 \begin{equation*}
\dive v(x)=\sum_{k:\;\Omega_{k}\ni x}\dive v_{k}(x)=\sum_{k:\;\Omega_{k}\ni x}F_{k}(x),
\end{equation*}
and since by Lemma~\ref{thm:partition} \textit{i)} and \textit{ii)} the last term is equal
to $F(x)$, and then $v$ is the aimed solution.
\end{proof}

\section*{Acknowledgments}
The research that led to the present paper was partially supported by a grant of the group GNAMPA of INdAM.

\def\ocirc#1{\ifmmode\setbox0=\hbox{$#1$}\dimen0=\ht0 \advance\dimen0
  by1pt\rlap{\hbox to\wd0{\hss\raise\dimen0
  \hbox{\hskip.2em$\scriptscriptstyle\circ$}\hss}}#1\else {\accent"17 #1}\fi}
  \def\cprime{$'$} \def\polhk#1{\setbox0=\hbox{#1}{\ooalign{\hidewidth
  \lower1.5ex\hbox{`}\hidewidth\crcr\unhbox0}}} \def\cprime{$'$}

\end{document}